\documentclass{birkjour}
\usepackage{amsfonts,amsmath,amsthm,amscd,amssymb,latexsym,cite,verbatim,texdraw,floatflt,pb-diagram}

\usepackage{hyperref}
\hypersetup{
    colorlinks=true,
     linkcolor=blue,
     urlcolor = blue,
    }

\usepackage{tikz} % для рисунков

 \newtheorem{thm}{Theorem}[section]
 \newtheorem{cor}[thm]{Corollary}
 
 \newtheorem{prop}[thm]{Proposition}
 \theoremstyle{definition}
 \newtheorem{defn}[thm]{Definition}
 \theoremstyle{remark}
 \newtheorem{rem}[thm]{Remark}
 \newtheorem*{ex}{Example}
 \numberwithin{equation}{section}

\newcommand{\ve}{\varepsilon}
\newcommand{\RR}{\mathbb{R}}
\newcommand{\NN}{\mathbb{N}}
\begin{document}

%-------------------------------------------------------------------------
% editorial commands: to be inserted by the editorial office
%
%\firstpage{1} \volume{228} \Copyrightyear{2004} \DOI{003-0001}
%\seriesextraline{This is the Concrete Title of this Book\br H.E. R and S.T.C. W, Eds.}
%\firstpage{1}

\title[On three-point generalizations]{On three-point generalizations of Banach and Edelstein fixed point theorems}

%----------Author 1
\author[Christian Bey]{Christian Bey}

\address{
Institute of Mathematics, University of L\"{u}beck\\
Ratzeburger Allee 160\\
D-23562 L\"{u}beck\\
Germany}

\email{ch.bey@uni-luebeck.de}

%----------Author 2
\author[Evgeniy Petrov]{Evgeniy Petrov}

\address{
Institute of Applied Mathematics and Mechanics\\
of the NAS of Ukraine\\
Dobrovolskogo str. 1\\
84100 Slovyansk\\
Ukraine}

\email{eugeniy.petrov@gmail.com}

%\thanks{This work was completed with the support of our \TeX-pert.}
%----------Author 3
\author{Ruslan Salimov}
\address{Institute of Mathematics of the NAS of Ukraine\\
Tereschenkivska str. 3\\
01024 Kiev\\
Ukraine}

\email{ruslan.salimov1@gmail.com}



%\date{July 30, 2022}

\begin{abstract}
Let $X$ be a metric space. Recently in~\cite{P23} it was considered a new type of mappings $T\colon X\to X$ which can be characterized as mappings contracting perimeters of triangles. These mappings are defined by the condition based on the mapping of three points of the space instead of two, as it is adopted in many fixed point theorems.
In the present paper we consider so-called $(F,G)$-contracting mappings, which form a more general class of mappings than mappings contracting perimeters of triangles. The fixed point theorem for these mappings is proved. We prove also a fixed point theorem for mappings contracting perimeters of triangles in the sense of Edelstein.

\end{abstract}

%%% ----------------------------------------------------------------------
\maketitle
%%% ----------------------------------------------------------------------
%\tableofcontents

\section{Introduction}
The Contraction Mapping Principle was established by S. Banach in his dissertation (1920) and published in 1922~\cite{Ba22}.  It has been generalized in many ways over the years. It is possible to distinguish two types of generalizations of this theorem: in the first case the contractive nature of the mapping is weakened, see, e.g.
~\cite{Ki03,Wa12,Pr20,Po21};
%~\cite{BW69,Ci74,Ki03,Mk69,Ra62,Re72,Su06,Wa12,Pr20,Po21};
in the second case the topology is weakened, see, e.g.~\cite{DBC12,CS12,JKR11,KS13,LS12,Tu12,SIIR20}.
%~\cite{DBC12,Ba00,CS12,DD07,Fr00,JKR11,KKR90,KS13,LS12,Sa10,SRR09,Ta74,Tu12,SIIR20}

Let $X$ be a nonempty set. Recall that a mapping  $d\colon X\times X\to \mathbb{R}^+$, $\mathbb{R}^+=[0,\infty)$ is a \emph{metric} if for all $x,y,z \in X$ the following axioms hold:
\begin{itemize}
  \item [(i)] $(d(x,y)=0)\Leftrightarrow (x=y)$,
  \item [(ii)] $d(x,y)=d(y,x)$,
  \item [(iii)] $d(x,y)\leqslant d(x,z)+d(z,y)$.%
\end{itemize}
The pair $(X,d)$ is called a \emph{metric space}.

In~\cite{P23} it was considered a new type of mappings $T\colon X\to X$ which can be characterized as mappings contracting perimeters of triangles. %, see Definition~\ref{d1}.
\begin{defn}\label{d1}
Let $(X,d)$ be a metric space with $|X|\geqslant 3$. We shall say that $T\colon X\to X$ is a \emph{mapping contracting perimeters of triangles} on $X$ if there exists $\alpha\in [0,1)$ such that the inequality
  \begin{equation}\label{e1}
   d(Tx,Ty)+d(Ty,Tz)+d(Tx,Tz) \leqslant \alpha (d(x,y)+d(y,z)+d(x,z))
  \end{equation}
  holds for all three pairwise distinct points $x,y,z \in X$.
\end{defn}

Recall that a mapping $T\colon X\to X$ is a \emph{contraction} on the metric space $(X,d)$ if there exists $\alpha\in [0,1)$ such that
\begin{equation}\label{e3}
d(Tx,Ty)\leqslant \alpha d(x,y)
\end{equation}
for all $x,y \in X$.

It is clear that every contraction is a mapping contractive perimeters of triangles.

\begin{rem}
Note that the requirement for $x,y,z\in X$ to be pairwise distinct in Definition~\ref{d1} is essential. One can see that otherwise this definition is equivalent to the definition of contraction mapping.
\end{rem}

In~\cite{P23} it was shown that mappings contracting perimeters of triangles are continuous. The fixed point theorem for such mappings was proved and the classical Banach fixed-point theorem was obtained like a simple corollary.  An example of a mapping contracting perimeters of triangles which is not a contraction mapping was constructed for a space $X$ with $\operatorname{card}(X)=\aleph_0$.

In~\cite{KS14} authors noted that except Banach's fixed point theorem there are also three classical fixed point theorems against which metric extensions are usually checked. These are, respectively, Nadler's well-known set-valued extension of Banach's theorem~\cite{Na69}, the extension of Banach's theorem to nonexpansive mappings~\cite{Ki65}, and Caristi's theorem~\cite{Ca76}. Note that an important place in the fixed point theory is also occupied by Edelstein's~\cite{Ed62} fixed point theorem, the scheme of the proof of which is fundamentally different from the proof of above mentioned theorems.

\begin{thm}[Edelstein, 1962]\label{t13}
Let $X$ be a metric space and let $T\colon X\to X$ be a mapping satisfying
\begin{equation}\label{e31}
  d(Tx,Ty)<d(x,y)
\end{equation}
for all $x\neq y$, $x,y \in X$. Assume that there exists $x\in X$ such that the sequence of iterates $(T^nx)$ contains a subsequence $(T^{n_k}x)$ convergent to a point $\xi \in X$. Then $\xi$ is a unique fixed point of $T$.
\end{thm}
Clearly, that if $X$ is a compact metric space and $T\colon X\to X$ satisfies ~(\ref{e31}) for all $x\neq y$, $x,y \in X$, then there exists a unique fixed point. Recall that mappings of type~(\ref{e31}) are called \emph{contractive}.

One new interesting proof of Edelstein's theorem was given in~\cite{DK98}. Note that generalizations of this theorem are not as numerous as generalizations of Banach's theorem. One of the most famous generalizations is Suzuki's~\cite{Su09} theorem. Let us mention also generalizations of Theorem~\ref{t13} in topological spaces~\cite{Li80, Ja75}, $\nu$-generalized metric spaces~\cite{Suz18,SABM15,Suz16}, complete metric spaces~\cite{Suz17}, compact metric spaces~\cite{Suz19}, Cartesian product of metric spaces~\cite{Sz76}. See~\cite{Ka12,DKR12,RSS79,DIK09,MSV13,SV15,Cze24} for further developments in this direction.

In Section~\ref{sec2} we consider so-called $(F,G)$-contracting mappings, which form a more general class of mappings than mappings contracting perimeters of triangles, see Definition~\ref{d4}. We show that $(F,G)$-contracting mappings are continuous and prove the fixed point theorem for these mappings.

In Section~\ref{sec3} we show continuity and prove a fixed point theorem for mappings contracting perimeters of triangles in the sense of Edelstein.

\section{Mappings with controlled contraction}\label{sec2}

In this section we consider a more general class of mappings in ordinary metric spaces than mappings contracting perimeters of triangles and prove a fixed point theorem for this class.

Note that Definition~\ref{d4} and Theorem~\ref{t2} were initially presented at the International scientific online conference ``Algebraic and geometric methods of analysis''~\cite{PS23}.

\begin{defn}\label{d4}
Let $(X,d)$ be a metric space with $|X|\geqslant 3$ and let functions $F,G\colon \RR^+\times\RR^+\times\RR^+\to\RR^+$ be such that for all $\xi, \eta,\zeta\in\RR^+$ the following conditions hold:

\begin{equation}\label{u1}
\begin{split}
    F(\eta,\xi,\zeta)=F(\xi,\eta,\zeta)=F(\xi,\zeta,\eta),\\
    G(\eta,\xi,\zeta)=G(\xi,\eta,\zeta)=G(\xi,\zeta,\eta);
\end{split}
\end{equation}
\begin{equation}\label{u2}
 G(\xi,\eta,\zeta)\geqslant \xi;
\end{equation}
\begin{equation}\label{u3}
   F(\xi,\eta,\zeta)\geqslant G(\xi,\eta,\zeta);
\end{equation}
%\begin{equation}\label{u4}
%   F \text{ is strictly increasing in each variable},
%\end{equation}
\begin{equation}\label{u44}
 G(0,0,0)=0 \text{ and }  G \text{ is continuous at } (0,0,0);
\end{equation}
\begin{equation}\label{u45}
\text{The function } G \text{ is monotone increasing in all of its arguments.}
\end{equation}

We shall say that $T\colon X\to X$ is an $(F,G)$-\emph{contracting mapping} on $X$ if there exists $\alpha\in [0,1)$ such that the inequality
  \begin{equation}\label{u5}
   F(d(Tx,Ty),d(Ty,Tz),d(Tx,Tz)) \leqslant \alpha\,G(d(x,y),d(y,z),d(x,z))
  \end{equation}
  holds for all three pairwise distinct points $x,y,z \in X$.
\end{defn}

\begin{rem}
Inequality~(\ref{u5}) and condition~(\ref{u44}) imply
$$
F(0,0,0)=0.
$$
\end{rem}

\begin{rem}
Note that~(\ref{u2}) and~(\ref{u3}) imply
 \begin{equation}\label{uu2}
   F(\xi,\eta,\zeta)\geqslant \xi
\end{equation}
for all $\xi,\eta,\zeta \in \mathbb R^+$.
\end{rem}

\begin{rem}
Note that the requirement for $x,y,z\in X$ to be pairwise distinct is essential. One can see that, otherwise, in the case
$$
 F(\xi,\eta,\zeta)=G(\xi,\eta,\zeta)=\xi+\eta+\zeta
$$
(when $F$ and $G$ are such that $T$ is a mapping contracting perimeters of triangles) this definition is equivalent to Definition~\ref{d1}.
\end{rem}

\begin{prop}
If in Definition~\ref{d4} the function $G$ additionally satisfies the condition
\begin{equation}\label{u22}
 G(\xi,\eta,\zeta)\leqslant K\xi \, \text{ with }\,  K\alpha<1,
\end{equation}
then $T$ is a mapping contracting perimeters of triangles.
\end{prop}
\begin{proof}
Indeed, by~(\ref{u1}) and~(\ref{uu2}) we have
$$
   F(\xi,\eta,\zeta)\geqslant \xi, \quad
   F(\xi,\eta,\zeta)\geqslant \eta, \quad
      F(\xi,\eta,\zeta)\geqslant \zeta.
$$
Hence,
$$
 \frac{\xi+\eta+\zeta}{3} \leqslant F(\xi,\eta,\zeta).
$$
Analogously, by~(\ref{u1}) and~(\ref{u22}) we have
$$
  G(\xi,\eta,\zeta)\leqslant  K\frac{\xi+\eta+\zeta}{3}.
$$
Finally, using the last two inequalities in~(\ref{u5}), we get the desired assertion.
\end{proof}

\begin{prop}\label{p36}
Let $(X,d)$ be a metric space, $|X|\geqslant 3$, and let $T\colon X\to X$ be an $(F,G)$-contracting mapping on $X$. Then $T$ is continuous.
\end{prop}

\begin{proof}
Let $(X,d)$ be a metric space with $|X|\geqslant 3$, $T\colon X\to X$ be a mapping contracting perimeters of triangles on $X$ and let $x_0$ be an isolated point in $X$. Then, clearly, $T$ is continuous at $x_0$. Let now $x_0$ be an accumulation point. Let us show that for every $\ve>0$, there exists $\delta>0$ such that $d(Tx_0,Tx)<\ve$ whenever $d(x_0,x)<\delta$. Suppose that $x\neq x_0$, otherwise this assertion is evident.
Since $x_0$ is an accumulation point, for every $\delta>0$ there exists $y\in X$ such that $x_0\neq y\neq x$ and $d(x_0,y)< \delta$. Since the points $x_0$, $x$ and $y$ are pairwise distinct by~(\ref{uu2}) and~(\ref{u5})   we have
$$
d(Tx_0,Tx)\leqslant F(d(Tx_0,Tx),d(Tx_0,Ty),d(Tx,Ty))
$$
$$
\leqslant \alpha G(d(x_0,x),d(x_0,y),d(x,y)).
$$
Further, using the triangle inequality $d(x,y) \leqslant d(x_0,x)+d(x_0,y)$ and~(\ref{u45}), we get
$$
d(Tx_0,Tx)\leqslant \alpha G(d(x_0,x),d(x_0,y),d(x_0,x)+d(x_0,y))
\leqslant
\alpha G(\delta,\delta,2\delta).
$$
Finally, using~(\ref{u44}), we get that for every $\ve>0$ there exists $\delta>0$ such that the inequality
$\alpha G(\delta,\delta,2\delta)<\ve$ holds, which completes the proof.
\end{proof}

Let $T$ be a mapping on the metric space $X$. A point $x\in X$ is called a \emph{periodic point of period $n$}, $n\in \mathbb N\setminus\{0\}$, if $T^n(x) = x$. The least positive integer $n$ for which $T^n(x) = x$ is called the prime period of $x$, see, e.g.,~\cite[p.~18]{De22}. In particular, the point $x$ is of prime period $2$ if $T(T(x))=x$ and $Tx\neq x$.

\begin{thm}\label{t2}
Let $(X,d)$, $|X|\geqslant 3$, be a complete metric space and let  $T\colon X\to X$ be a mapping satisfying the following two conditions:
\begin{itemize}
  \item [(i)] $T$ does not possess periodic points of prime period $2$.
  \item [(ii)] $T$ is an $(F,G)$-contracting mapping on $X$.
\end{itemize}
Then $T$ has a fixed point. The number of fixed points is at most two.
\end{thm}

\begin{proof}
Let $x_0\in X$. Consider the sequence  $x_1=Tx_0$, $x_2=Tx_1$, \ldots, $x_{n+1}=Tx_n$, \ldots . Suppose first that $x_i$ is not a fixed point of the mapping $T$ for every $i=0,1,...$. Let us show that all $x_i$ are different. Since $x_i$ is not fixed, then $x_i\neq x_{i+1}=Tx_i$. By condition (i) $x_{i+2}=T(T(x_i))\neq x_i$ and by the assumption that $x_{i+1}$ is not fixed we have $x_{i+1}\neq x_{i+2}=Tx_{i+1}$. Hence, $x_i$, $x_{i+1}$ and $x_{i+2}$ are pairwise distinct.

Further, set
$$
\tilde{p}_0=F(d(x_0,x_1),d(x_1,x_2),d(x_2,x_0)), \, \, p_0=G(d(x_0,x_1),d(x_1,x_2),d(x_2,x_0)),
$$
$$
\tilde{p}_1=F(d(x_1,x_2),d(x_2,x_3),d(x_3,x_1)), \, \,
p_1=G(d(x_1,x_2),d(x_2,x_3),d(x_3,x_1)),
$$
$$
\cdots
$$
$$
\tilde{p}_n=F(d(x_n,x_{n+1}),d(x_{n+1},x_{n+2}),d(x_{n+2},x_n)), $$
$$
p_n=G(d(x_n,x_{n+1}),d(x_{n+1},x_{n+2}),d(x_{n+2},x_n)),
$$
$$
\cdots .
$$
Then by~(\ref{u5}) we have $\tilde{p}_1\leqslant \alpha p_0$, $\tilde{p}_2\leqslant \alpha p_1$, \ldots, $\tilde{p}_n\leqslant \alpha p_{n-1}$.
Hence and by condition~(\ref{u3}) we have
\begin{equation}\label{e22}
p_0>\alpha p_0\geqslant \tilde{p}_1\geqslant p_1>\alpha p_1 \geqslant \tilde{p}_2\geqslant p_2 \ldots .
\end{equation}
Consequently,
\begin{equation}\label{e21}
p_0>p_1>...>p_n>\ldots .
\end{equation}
Suppose now that $j\geqslant 3$ is a minimal natural number such that $x_j=x_i$ for some $i$ such that $0\leqslant i<j-2$. Then  $x_{j+1}=x_{i+1}$, $x_{j+2}=x_{i+2}$. Hence, $p_i=p_j$ which contradicts to~(\ref{e21}). Thus, all $x_i$ are different.

Further, let us show that $\{x_i\}$ is a Cauchy sequence.  By~(\ref{u1}), (\ref{u2}) and~(\ref{e22})  we have
$$
d(x_1,x_2)\leqslant p_0,
$$

$$
d(x_2,x_3)\leqslant p_1\leqslant \alpha p_0,
$$

$$
d(x_3,x_4)\leqslant p_2\leqslant \alpha p_1\leqslant  \alpha^2 p_0,
$$
$$
\cdots
$$
$$
d(x_n,x_{n+1})\leqslant p_{n-1}\leqslant \alpha^{n-1} p_0,
$$
$$
d(x_{n+1},x_{n+2})\leqslant p_{n}\leqslant \alpha^{n} p_0,
$$
$$
\cdots .
$$
By the triangle inequality,
$$
d(x_n,\,x_{n+p})\leqslant d(x_{n},\,x_{n+1})+d(x_{n+1},\,x_{n+2})+\ldots+d(x_{n+p-1},\,x_{n+p})
$$
$$
\leqslant \alpha^{n-1}p_0+\alpha^{n}p_0+\cdots +\alpha^{n+p-2}p_0 = \alpha^{n-1}(1+\alpha+\ldots+\alpha^{p-1})p_0
=\alpha^{n-1}\frac{1-\alpha^{p}}{1-\alpha}p_0.
$$
Since $0\leqslant\alpha<1$ we have $0\leqslant\alpha^p<1$ and
$$
d(x_n,\,x_{n+p})\leqslant \frac{p_0\alpha^{n-1}}{1-\alpha}.
$$
Hence, $d(x_n,\,x_{n+p})\to 0$ as $n\to \infty$ for every $p>0$. Thus, $\{x_n\}$ is a Cauchy sequence. By completeness of $(X,d)$, this sequence has a limit $x^*\in X$.

Let us prove that $Tx^*=x^*$. Since $x_n\to x^*$, and by Proposition~\ref{p36} the mapping $T$ is continuous, we have $x_{n+1}=T x_n\to Tx^*$.  By triangle inequality we have
$$
d(x^*,Tx^*)\leqslant d(x^*,x_{n})+d(x_{n},Tx^*)
\to 0
$$
as $n\to \infty$, which means that $x^*$ is the fixed point.

Suppose that there exist at least three pairwise distinct fixed points $x$, $y$ and $z$.  Then $Tx=x$, $Ty=y$ and $Tz=z$, which contradicts to~(\ref{u5}) combined with~(\ref{u3}).
\end{proof}

\begin{ex}
It is easy to see that the class of functions $F$ and $G$ satisfying the conditions of Definition~\ref{d4} is enough wide. As an example we can take $F(\xi,\eta,\zeta)=\xi+\eta+\zeta$ and $G(\xi,\eta,\zeta)=(\xi^{q}+\eta^{q}+\zeta^{q})^{\frac{1}{q}}$, $q\geqslant 1$.
\end{ex}

\begin{ex}
Let $G(\xi,\eta,\zeta)=\xi+\eta+\zeta$ and let
$$
F(\xi,\eta,\zeta)=3\varphi^{-1}\left(\frac{\varphi(\xi)+\varphi(\eta)+\varphi(\zeta)}{3}\right),
$$
where $\varphi\colon[0,\infty)\to [0,\infty)$ is continuous, strictly increasing and  convex function. Since $\varphi$ is convex by Jensen's inequality we have
$$
\varphi\left(\frac{\xi+\eta+\zeta}{3}\right)\leqslant \frac{\varphi(\xi)+\varphi(\eta)+\varphi(\zeta)}{3}
$$
for all $\xi,\eta,\zeta \in [0,\infty)$, which implies condition~(\ref{u3}). The other conditions of Definition~\ref{d4} can be easily verified by the reader.
\end{ex}

\begin{ex}
Let the function $H$ satisfy all conditions of Definition~\ref{d4} that satisfies the function $G$, then as $F$ we can take $F(\xi,\eta,\zeta)=G(\xi,\eta,\zeta)+H(\xi,\eta,\zeta)$.
\end{ex}

\begin{rem}
Let the supposition of Theorem~\ref{t2} holds, let additionally $F$ and $G$ be continuous at their domain and let the mapping $T$ have a fixed point $x^*$ which is a limit of some iteration sequence $x_0, x_1=Tx_0, x_2=Tx_1,\ldots$ such that $x_n\neq x^*$ for all $n=1,2,\ldots$. Then $x^*$ is a unique fixed point.

Indeed, suppose that $T$ has another fixed point $x^{**}\neq x^*$.
It is clear that $x_n\neq x^{**}$ for all $n=1,2,\ldots$. Hence, we have that the points $x^*$, $x^{**}$ and $x_n$ are pairwise distinct for all $n=1,2,\ldots$. Consider the ratio
$$
R_n=\frac{F(d(Tx^*,Tx^{**}),d(Tx^*,Tx_{n}),d(Tx^{**},Tx_{n}))}{G(d(x^*,x^{**}),d(x^*,x_{n}),d(x^{**},x_{n}))}
$$
$$
=\frac{F(d(x^*,x^{**}),d(x^*,x_{n+1}),d(x^{**},x_{n+1}))}{G(d(x^*,x^{**}),d(x^*,x_{n}),d(x^{**},x_{n}))}.
$$
It is clear that $d(x^*,x_{n+1})\to 0$, $d(x^*,x_{n})\to 0$, $d(x^{**},x_{n+1})\to d(x^{**},x^*)$ and $d(x^{**},x_{n})\to d(x^{**},x^*)$. Taking into consideration condition~(\ref{u3}),
suppose first that
$$
F(d(x^*,x^{**}),0,d(x^{**},x^*))=G(d(x^*,x^{**}),0,d(x^{**},x^*)).
$$
In this case by continuity of $F$ and $G$ we get $R_n\to 1$ as $n\to \infty$.
Now suppose that
$$
F(d(x^*,x^{**}),0,d(x^{**},x^*))> G(d(x^*,x^{**}),0,d(x^{**},x^*)).
$$
In this case by continuity of $F$ and $G$ we obtain
$R_n> 1$ for some sufficiently large $n$. Anyway both cases contradict to condition~(\ref{u5}).
\end{rem}

Recall that for a given metric space $X$, a point $x \in X$ is said to be an \emph{accumulation point} of  $X$ if every open ball centered at $x$ contains infinitely many points of $X$.

\begin{prop}\label{p3}
Let $(X,d)$ be a metric space and let $T\colon X\to X$ be an $(F,G)$-contracting mapping on $X$ with continuous $F$ and $G$ and with
\begin{equation}\label{g1}
  G(\xi,\xi,0) \leqslant k\xi,
\end{equation}
where $k$ is such that $\alpha k<1$ and $\alpha$ is as in~(\ref{u5}).
If $x$ is an accumulation point of $X$, then inequality~(\ref{e3}) holds for all points $y\in X$ with the coefficient $\alpha k$.
\end{prop}
\begin{proof}
Let $y\in X$ and let $x\in X$ be an accumulation point. Then there exists a sequence $(z_n)$ such that $z_n\to x$, $z_n\neq y$, $z_n\neq x$, and all $z_n$ are different.
Hence, by~(\ref{u5}) the inequality
  \begin{equation*}
   F(d(Tx,Ty),d(Ty,Tz_n),d(Tx,Tz_n)) \leqslant \alpha G(d(x,y),d(y,z_n),d(x,z_n))
  \end{equation*}
holds for every $n\in \NN^+$. Since every metric is continuous, we have $d(y,z_n) \to d(x,y)$, $d(x,z_n)\to 0$, $d(Ty,Tz_n)\to d(Tx,Ty)$ and  $d(Tx,Tz_n)\to 0$. Letting $n\to \infty$ and using the continuity of $F$ and $G$ we obtain
  \begin{equation*}
  F(d(Tx,Ty),d(Tx,Ty),0) \leqslant \alpha G(d(x,y),d(x,y),0).
  \end{equation*}
By~(\ref{uu2}) we have
  \begin{equation*}
d(Tx,Ty) \leqslant \alpha G(d(x,y),d(x,y),0).
  \end{equation*}
Using~(\ref{g1}), we obtain~(\ref{e3}) with another coefficient $\alpha k$.
\end{proof}
\begin{cor}\label{cor3}
Let the supposition of Proposition~\ref{p3} hold. If all points of $X$ are accumulation points, then $T$ is a contraction mapping.
\end{cor}

\section{Edelstein type fixed point theorem}\label{sec3}

The main result of this section is an analogue of Edelstein's fixed point theorem for mappings contracting perimeters of triangles.

\begin{defn}\label{dd1}
Let $(X,d)$ be a metric space with $|X|\geqslant 3$. We shall say that $T\colon X\to X$ is a mapping \emph{contracting perimeters of triangles in the sense of Edelstein} on $X$ if the inequality
  \begin{equation}\label{ee1}
   d(Tx,Ty)+d(Ty,Tz)+d(Tx,Tz) < d(x,y)+d(y,z)+d(z,x)
  \end{equation}
  holds for all three pairwise distinct points $x,y,z \in X$.

\end{defn}

\begin{prop}
Mappings contracting perimeters of triangles in the sense of Edelstein are continuous.
\end{prop}

\begin{proof}
Let $(X,d)$ be a metric space with $|X|\geqslant 3$, $T\colon X\to X$ be a mapping contracting perimeters of triangles on $X$ and let $x_0$ be an isolated point in $X$. Then, clearly, $T$ is continuous at $x_0$. Let now $x_0$ be an accumulation point. Let us show that for every $\ve>0$, there exists $\delta>0$ such that $d(Tx_0,Tx)<\ve$ whenever $d(x_0,x)<\delta$. Let $x\neq x_0$. Since $x_0$ is an accumulation point, for every $\delta>0$ there exists $y\in X$, $x_0\neq y\neq x$, such that $d(x_0,y)< \delta$. By~(\ref{ee1}) we have
$$
d(Tx_0,Tx)\leqslant d(Tx_0,Tx)+d(Tx_0,Ty)+d(Tx,Ty)
$$
$$
<d(x_0,x)+d(x_0,y)+d(x,y).
$$
Using the triangle inequality $d(x,y) \leqslant d(x_0,x)+d(x_0,y)$, we have
$$
d(Tx_0,Tx) < 2(d(x_0,x)+d(x_0,y))<2(\delta + \delta)=4\delta.
$$
Setting $\delta=\ve / 4$, we obtain the desired inequality.
\end{proof}

\begin{thm}\label{t1}
Let $(X,d)$, $|X|\geqslant 3$, be a metric space and let the mapping $T\colon X\to X$ satisfy the following three conditions:
\begin{itemize}
  \item [(i)] $T$ is a mapping contracting perimeters of triangles in the sense of Edelstein.
  \item [(ii)]  $T$ does not possess periodic points of prime period $2$.
   \item [(iii)] %\textcolor[rgb]{0.00,0.00,0.50}{There exists an $x\in X$ such that the sequence $(T^n(x))$ has a subsequence $(T^{n_i}(x))$ which converges to $\xi \in X$. }
       There exists an $x\in X$ such that the sequence $(\xi_n)$, $\xi_n=T^nx$, has a subsequence $(\xi_{n_i})$ which converges to $\xi \in X$.
\end{itemize}
Then $\xi$ is a fixed point of $T$. The general number of fixed points of $T$ is at most two.

\end{thm}

\begin{proof}
Suppose that $\xi$ is not a fixed point of $T$. Then, by condition (ii), there are only two possibilities:  1) the points $\xi$, $T\xi$ and $T^2\xi$ are pairwise distinct; 2) $T\xi$ is fixed.

Suppose that possibility 1) holds. Throughout the text below for any $x,y,z\in X$ denote the perimeter $p(x,y,z)$ on the  points $x,y,z$ by
$$
p(x,y,z)=d(x,y)+d(y,z)+d(z,x).
$$
The mapping $r(q,s,t)$ of%qst
$$
Y:=X\times X \times X\setminus(\{(x,y,z):x=y\}\cup \{(x,y,z):y=z\}\cup \{(x,y,z):z=x\})
$$
into the real line, defined by
$$
r(p,q,s):=\frac{p(Tq,Ts,Tt)}{p(q,s,t)}
$$
is continuous on $Y$ since every metric is continuous and the denominator of the fraction is nonzero. Hence, there exists a neighbourhood $U$ of $(\xi,T(\xi),T^2(\xi))$ and $R$ such that $(q,s,t)\in U$ implies
\begin{equation}\label{e11}
0 \le r(q,s,t)<R<1.
\end{equation}
Let $S_1$, $S_2$ and $S_3$ be open balls centered at $\xi$, $T\xi$ and $T^2\xi$, respectively, of radius $\rho >0$ small enough so as to have
\begin{equation}\label{eq1}
\rho < \frac{1}{7}p(\xi,T\xi,T^2\xi)
\end{equation}
and $S_1\times S_2 \times S_3 \subseteq U$.
Consider the sequences $(\xi_{n_i+1})$ and $(\xi_{n_i+2})$. Since $T$ is continuous, by condition (iii) the first sequence converges to $T\xi$ and the second one converges to $T^2\xi$. Hence, and since $\xi_{n_i}\to \xi$, there exists a positive integer $N$ such that $i>N$ implies $\xi_{n_i}\in S_1$, $\xi_{n_i+1}\in S_2$ and $\xi_{n_i+2}\in S_3$. Then

\begin{equation}\label{ee3}
p(\xi_{n_i},\xi_{n_i+1},\xi_{n_i+2})>\rho, \quad (i>N).
\end{equation}
Indeed, assume the opposite. Then, by the triangle inequality,
$$
p(\xi,T\xi,T^2\xi) \le d(\xi,\xi_{n_i})+d(\xi_{n_i},\xi_{n_i+1})+d(\xi_{n_i+1}, T\xi)
$$
$$
+d(T\xi,\xi_{n_i+1}))+d(\xi_{n_i+1},\xi_{n_i+2})+d(\xi_{n_i+2}, T^2\xi)
$$
$$
+d(T^2\xi_{n_i+2}),\xi_{n_i+2})+d(\xi_{n_i+2},\xi_{n_i})+d(\xi_{n_i},\xi)
$$
$$
\le 6\rho + p(\xi_{n_i},\xi_{n_i+1},\xi_{n_i+2}) \le 7 \rho,
$$
which contradicts to~(\ref{eq1}).

Suppose first that $\xi_i$ is a fixed point of $T$ for some $i$. Clearly, in this case every subsequence of $(\xi_n)$ converges to $\xi_i=\xi$. Suppose now that $\xi_i$ is not a fixed point of the mapping $T$ for every $i=1,2,...$. Since $\xi_i$ is not fixed, then $\xi_i\neq \xi_{i+1}=T\xi_i$. Since $T$ have no periodic points of prime period $2$ we have $\xi_{i+2}=T(T\xi_i)\neq \xi_i$ and by the supposition that $\xi_{i+1}$ is not fixed we have $\xi_{i+1}\neq \xi_{i+2}=T\xi_{i+1}$. Hence, every three consecutive points $\xi_i$, $\xi_{i+1}$ and $\xi_{i+2}$ are pairwise distinct.

Let us show that all $\xi_i$ are different. Set
$$
p_0=p(x_0,x_1,x_2),\quad p_1=p(x_1,x_2,x_3), \cdots, p_n=p(x_n,x_{n+1},x_{n+2}), \cdots .
$$
Since  $\xi_i$, $\xi_{i+1}$ and $\xi_{i+2}$  are pairwise distinct by~(\ref{ee1}) we have \begin{equation}\label{e2}
p_0>p_1>...>p_n>\ldots .
\end{equation}
Suppose that $j\geqslant 3$ is a minimal natural number such that $\xi_j=\xi_i$ for some $i$ such that $0\leqslant i<j-2$. Then  $\xi_{j+1}=\xi_{i+1}$, $\xi_{j+2}=\xi_{i+2}$. Hence, $p_i=p_j$ which contradicts to~(\ref{e2}).

Since $\xi_{n_i}, \xi_{n_i+1}$ and $\xi_{n_i+2}$ are pairwise distinct points of $(\xi_i)$, for $i>N$ we can use~(\ref{e11}):
$$
\begin{aligned}
 p(\xi_{n_i+1},\xi_{n_i+2},\xi_{n_i+3})
< R \cdot p(\xi_{n_i},\xi_{n_i+1},\xi_{n_i+2}).
\end{aligned}
$$
Repeated use of condition~(\ref{ee1}) and this inequality gives for $\ell>j>N$
$$
p(\xi_{n_{\ell}},\xi_{n_{\ell}+1},\xi_{n_{\ell}+2})
\leqslant  p(\xi_{n_{\ell-1}+1},\xi_{n_{\ell-1}+2},\xi_{n_{\ell-1}+3})
$$
$$
< R\cdot p(\xi_{n_{\ell-1}},\xi_{n_{\ell-1}+1},\xi_{n_{\ell-1}+2}) \leqslant \ldots
$$
$$
< R^{\ell-j} \cdot p(\xi_{n_{j}},\xi_{n_{j}+1},\xi_{n_{j}+2}) \rightarrow 0,\quad \ell\rightarrow \infty,
$$
which contradicts to~(\ref{ee3}). Thus, $T(\xi)=\xi$.

Suppose that possibility 2) holds.
Since $T\xi_{n_i}=\xi_{n_i+1}\to T\xi$, for every $\ve>0$ there exist $N_1>0$ such that for every $i,j>N_1$,  $i\neq j$, inequalities $d(\xi_{n_i+1},T\xi)<\ve$ and $d(\xi_{n_j+1},T\xi)<\ve$ hold. Hence, using the triangle inequality
$
d(\xi_{n_i+1},\xi_{n_j+1})\leqslant  d(\xi_{n_i+1},T\xi)+d(T\xi,\xi_{n_j+1}),
$
we get
\begin{equation}\label{e13}
p(\xi_{n_i+1},\xi_{n_j+1},T\xi)<4\ve.
\end{equation}
Since there is no fixed points in the sequence $(\xi_n)$ and $T\xi$ is fixed for $T$, we have that $T\xi\neq\xi_n$ for all $n\geqslant 1$.
Hence, the points $\xi_{n_{i+1}}$, $\xi_{n_{j+1}}$ and $T\xi$ are pairwise distinct and for any $k\geqslant 1$ we can consecutively apply $k$ times inequality~(\ref{ee1}) to inequality~(\ref{e13}):
\begin{equation}\label{e33}
p(T^k\xi_{n_i+1},T^k\xi_{n_j+1},T^kT\xi)
=p(\xi_{n_i+k+1},\xi_{n_j+k+1},T\xi)<4\ve.
\end{equation}

On the other hand, since $\xi_{n_i}\to \xi$, for every $\ve>0$ there exist $N_2>0$ such that for every $m>N_2$ the inequality $d(\xi_{n_m},\xi)<\ve$  holds. It is clear that it is always possible to choose $m>N_2$ such that $n_m>n_i+1$. Let $k=n_m-n_i-1$. Hence, $n_m=n_i+k+1$ and $\xi_{n_m}=\xi_{n_i+k+1}$.
Further, using the inequality $d(\xi,T\xi)\leqslant d(\xi, \xi_{n_m})+d(\xi_{n_m},\xi_{n_j+k+1})+d(\xi_{n_j+k+1},T\xi)$, we obtain
$$p(\xi_{n_i+k+1},\xi_{n_j+k+1},T\xi)=p(\xi_{n_m},\xi_{n_j+k+1},T\xi)$$
$$
=d(\xi_{n_m},\xi_{n_j+k+1})+d(\xi_{n_j+k+1},T\xi)+d(T\xi,\xi_{n_m})
$$
$$
\geqslant d(\xi,T\xi)-d(\xi,\xi_{n_m})-d(\xi_{n_j+k+1},T\xi)+d(\xi_{n_j+k+1},T\xi)+d(T\xi,\xi_{n_m})
$$
$$
> d(\xi,T\xi)-\ve.
$$
Setting $\ve= d(\xi,T\xi)/5$ we obtain that this inequality contradicts to~(\ref{e33}).

Suppose that there exists at least three pairwise distinct fixed points $x$, $y$ and $z$.  Then $Tx=x$, $Ty=y$ and $Tz=z$, which contradicts to~(\ref{ee1}).
\end{proof}

\begin{cor}\label{c1}
Let $(X,d)$, $|X|\geqslant 3$, be a compact metric space and let the mapping $T\colon X\to X$ satisfy the following two conditions:
\begin{itemize}
  \item [(i)] $T$ is a mapping contracting perimeters of triangles in the sense of Edelstein.
  \item [(ii)]  $T$ does not possess periodic points of prime period $2$.
\end{itemize}
Then  $T$ has a fixed point.
The general number of fixed points of $T$ is at most two.
\end{cor}
\begin{proof}
Indeed, condition (iii) of Theorem~\ref{t1} holds if $X$ is compact.
\end{proof}

\begin{prop}\label{p1}
Let $(X,d)$, $|X|\geqslant 3$, be a metric space and let  $T\colon X\to X$ be a mapping contracting perimeters of triangles in the sense of Edelstein. If $x$ is an accumulation point of  $X$, then inequality~(\ref{e31}) holds for all points $y\in X$, $y\neq x$.
\end{prop}
\begin{proof}
Let $x\in X$ be an accumulation point and let $y\in X$, $y\neq x$. Since $x$ is an accumulation point, then there exists a sequence $z_n\to x$ such that $z_n \neq x$, $z_n \neq y$ and all $z_n$ are different.
Hence, by~(\ref{ee1}) the inequality
  \begin{equation*}
   d(Tx,Ty)+d(Ty,Tz_n)+d(Tx,Tz_n) < d(x,y)+d(y,z_n)+d(x,z_n)
  \end{equation*}
holds for every $n\in \mathbb N$. Since $d(x,z_n)\to 0$ and every metric is continuous we have $d(y,z_n) \to d(x,y)$. Since $T$ is continuous, we have $d(Tx,Tz_n)\to 0$ and, consequently,  $d(Ty,Tz_n)\to d(Tx,Ty)$. Letting $n\to \infty$, we obtain
  \begin{equation*}
   d(Tx,Ty)+d(Tx,Ty) <d(x,y)+d(x,y),
  \end{equation*}
which is equivalent to~(\ref{e31}).
\end{proof}
%The following corollary establishes a very interesting property of mappings contracting perimeters of triangles.
\begin{cor}\label{cor1}
Let $(X,d)$, $|X|\geqslant 3$, be a metric space and let $T\colon X\to X$ be a mapping contracting perimeters of triangles in the sense of Edelstein. If all points of $X$ are accumulation points, then the mapping $T$ is contractive.
\end{cor}

\begin{ex}
Let us construct an example of the mapping $T$ contracting perimeters of triangles in the sense of Edelstein, which is not a contractive mapping. Let $X=\{x,y,z\}$, $d(x,y)=d(y,z)=d(x,z)=1$, and let $T\colon X\to X$ be such that $Tx= x$, $Ty= y$ and $Tz=x$. One can easily see that~(\ref{ee1}) holds and~(\ref{e31}) does not hold since $x$ and $y$ are fixed points.
\end{ex}

\begin{ex}
Let us construct an example of a mapping $T\colon X\to X$ contracting perimeters of triangles in the sense of Edelstein that is not a contractive mapping and that is not a mapping contracting perimeters of triangles for a metric space $X$ with $|X|=\aleph_0$.
Let $X=\{x^*, x_0,x_0',x_1,x_1',\ldots \}$ and let $\ve$ be a positive real number.

Define the metric $d$ on $X$ as follows:
$$
d(x,y)=
\begin{cases}
\frac{1}{i^2}, &\text{if } x=x_{i-1} \text{ or } x=x_{i-1}',\\
&\text{and }~  y=x_{i}, \text{ or } y=x_{i}',~ i=1,2,3,...,\\
\frac{\ve}{i^2}, &\text{if } x=x_i, \,  y=x_{i}'\\
&\text{or }  x=x_{i-1}, y=x_{i-1}',~ i=1,3,5,...,\\
d(x_i,x_{i+1})+\cdots+d(x_{j-1},x_j), &\text{if } x=x_i \text{ or }  x=x_i'\\ &\text{and }  y=x_{j} \text{ or } y=x_{j}', ~~ i+1<j,\\
\frac{\pi^2}{6}-d(x_0,x_i), &\text{if } x=x_i, \, y=x^*,\\
0, &\text{if } x=y,
\end{cases}
$$
see Figure~\ref{fig1}.
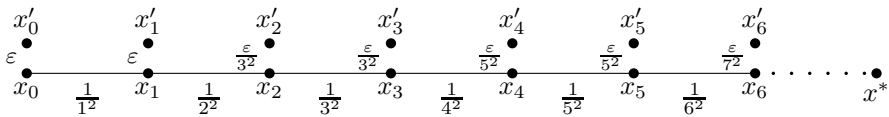
\begin{figure}[ht]
\begin{center}
\begin{tikzpicture}[scale=0.8]
\draw (1,0) node [below] {{$\frac{1}{1^2}$}};
\draw (3,0) node [below] {{$\frac{1}{2^2}$}};
\draw (5,0) node [below] {{$\frac{1}{3^2}$}};
\draw (7,0) node [below] {{$\frac{1}{4^2}$}};
\draw (9,0) node [below] {{$\frac{1}{5^2}$}};
\draw (11,0) node [below] {{$\frac{1}{6^2}$}};

\draw       (0,0.26) node [left] {\small{$\ve$}};
\draw       (2,0.26) node [left] {\small{$\ve$}};
\draw       (4,0.26) node [left] {\small{$\frac{\ve}{3^2}$}};
\draw       (6,0.26) node [left] {\small{$\frac{\ve}{3^2}$}};
\draw       (8,0.26) node [left] {\small{$\frac{\ve}{5^2}$}};
\draw       (10,0.26) node [left] {\small{$\frac{\ve}{5^2}$}};
\draw       (12,0.26) node [left] {\small{$\frac{\ve}{7^2}$}};

\draw       (0,0.5) node [above] {$x_0'$};
\draw       (2,0.5) node [above] {$x_1'$};
\draw       (4,0.5) node [above] {$x_2'$};
\draw       (6,0.5) node [above] {$x_3'$};
\draw       (8,0.5) node [above] {$x_4'$};
\draw       (10,0.5) node [above] {$x_5'$};
\draw       (12,0.5) node [above] {$x_6'$};
 \foreach \i in {(0,0.5),(2,0.5),(4,0.5),(6,0.5),(8,0.5),(10,0.5),(12,0.5)}
  \fill[black] \i circle (2.4pt);

\draw (0,0) node [below] {$x_0$} --
      (2,0) node [below] {$x_1$} --
      (4,0) node [below] {$x_2$} --
      (6,0) node [below] {$x_3$} --
      (8,0) node [below] {$x_4$} --
      (10,0) node [below] {$x_5$} --
      (12,0) node [below] {$x_6$};
\draw (14,0) node [below] {$x^*$};

 \foreach \i in {(0,0),(2,0),(4,0),(6,0),(8,0),(10,0),(12,0),(14,0)}
  \fill[black] \i circle (2.4pt);

 \foreach \i in {(12+0.3,0),(12+0.6,0),(12+0.9,0),(12+1.2,0),(12+1.5,0),(12+1.8,0)}
  \fill[black] \i circle (0.8pt);
\end{tikzpicture}
\begin{center}
\caption{The points of the space $(X,d)$ with distances between them.}\label{fig1}
\end{center}
\end{center}
\end{figure}

The reader can easily verify that for sufficiently small $\ve$ the metric $d$ is well defined. For this recall only the well-known fact $\sum_{i=1}^{\infty}\frac{1}{i^2}=\frac{\pi^2}{6}$.  Moreover, the space is complete with the single accumulation point $x^*$.

Define a mapping  $T\colon X\to X$ as $Tx_i=x_{i+1}$, $Tx_i'=x_{i+1}'$ for all $i=0,1,\ldots$ and $Tx^*=x^*$.
Since
$d(x_{i},x_{i}')=d(T{x_{i}},T{x_{i}'})$, $i=0,2,4,...$, using~(\ref{e31}), we see that $T$ is not a contractive mapping. Suppose that there exists $0\leqslant\alpha <1$ such that~(\ref{e1}) holds for all pairwise distinct $x,y,z \in X$. Let $x=x_i$, $y=x_{i+1}$, $z=x_{i+2}$. Consider the ratio
$$
\frac{d(Tx,Ty)+d(Ty,Tz)+d(Tx,Tz)}{d(x,y)+d(y,z)+d(x,z)}
$$
$$
=\frac{d(x_{i+1},x_{i+2})+d(x_{i+2},x_{i+3})+d(x_{i+1},x_{i+3})}
{d(x_{i},x_{i+1})+d(x_{i+1},x_{i+2})+d(x_{i},x_{i+2})}
$$
$$
={\left(\frac{2}{(i+2)^2}+\frac{2}{(i+3)^2}\right)}:
{\left(\frac{2}{(i+1)^2}+\frac{2}{(i+2)^2}\right)} \to 1
$$
as $i\to \infty$, which contradicts to~(\ref{e1}). Thus, $T$ is not a mapping contracting perimeters of triangles.
Verifying inequality~(\ref{ee1}) for all pairwise distinct $x,y,z\in X$ is almost evident. Thus, $T$ is a mapping contractive perimeters of triangles in the sense of Edelstein.
\end{ex}

\textbf{Acknowledgements.}  The second author was partially supported by the Volkswagen Foundation project ``From Modeling and Analysis to Approximation''.

%\bibliographystyle{unsrt}
%\bibliography{Xbib}

\end{document}